\documentclass[fleqn,12pt,twoside]{article}

\usepackage[headings]{espcrc1}
\readRCS
$Id: espcrc1.tex,v 1.2 2004/02/24 11:22:11 spepping Exp $
\usepackage{graphicx}
\usepackage[figuresright]{rotating}

\newtheorem{theorem}{Theorem}

\newtheorem{corollary}[theorem]{Corollary}

\newtheorem{lemma}[theorem]{Lemma}

\newtheorem{problem}{Problem}

\newenvironment{proof}[1][Proof.]{\begin{trivlist}
\item[\hskip \labelsep {\bfseries #1}]}{\end{trivlist}}

\newcommand{\AmS}{{\protect\the\textfont2
  A\kern-.1667em\lower.5ex\hbox{M}\kern-.125emS}}

\usepackage{amsmath}
\usepackage{amsfonts}
\usepackage{amssymb}
\title{Interval colorings of complete balanced multipartite graphs}

\author{Petros A. Petrosyan\address[MCSD]{Institute for Informatics and Automation Problems,\\
National Academy of Sciences, 0014, Armenia}%
\address{Department of Informatics and Applied Mathematics,\\
Yerevan State University, 0025, Armenia}%
\thanks{email: pet\_petros@\{ipia.sci.am, ysu.am, yahoo.com\}}}

\runtitle{Interval colorings of complete balanced multipartite
graphs}\runauthor{Petros A. Petrosyan}

\begin{document}

\maketitle

\begin{abstract}
A graph $G$ is called a complete $k$-partite ($k\geq 2$) graph if
its vertices can be partitioned into $k$ independent sets
$V_{1},\ldots,V_{k}$ such that each vertex in $V_{i}$ is adjacent to
all the other vertices in $V_{j}$ for $1\leq i<j\leq k$. A complete
$k$-partite graph $G$ is a complete balanced $k$-partite graph if
$\vert V_{1}\vert =\vert V_{2}\vert =\cdots =\vert V_{k}\vert$. An
edge-coloring of a graph $G$ with colors $1,\ldots,t$ is an interval
$t$-coloring if all colors are used, and the colors of edges
incident to each vertex of $G$ are distinct and form an interval of
integers. A graph $G$ is interval colorable if $G$ has an interval
$t$-coloring for some positive integer $t$. In this paper we show
that a complete balanced $k$-partite graph $G$ with $n$ vertices in
each part is interval colorable if and only if $nk$ is even. We also
prove that if $nk$ is even and $(k-1)n\leq t\leq
\left(\frac{3}{2}k-1\right)n-1$, then a complete balanced
$k$-partite graph $G$ admits an interval $t$-coloring. Moreover, if
$k=p2^{q}$, where $p$ is odd and $q\in \mathbb{N}$, then a complete
balanced $k$-partite graph $G$ has an interval $t$-coloring for each
positive integer $t$ satisfying
$(k-1)n\leq t\leq \left(2k-p-q\right)n-1$.\\


\end{abstract}

\section{Introduction}\

Throughout this paper all graphs are finite, undirected, and have no
loops or multiple edges. Let $V(G)$ and $E(G)$ denote the sets of
vertices and edges of $G$, respectively. For $F\subseteq E(G)$, the
subgraph obtained by deleting the edges of $F$ from $G$ is denoted
by $G-F$. The maximum degree of $G$ is denoted by $\Delta(G)$. The
terms and concepts that we do not define can be found in \cite{b19}.

An edge-coloring of a graph $G$ is a mapping $\alpha:E(G)\rightarrow
\mathbb{N}$. A proper edge-coloring of a graph $G$ is an
edge-coloring $\alpha$ of $G$ such that $\alpha(e)\neq
\alpha(e^{\prime})$ for any pair of adjacent edges $e,e^{\prime}\in
E(G)$. The edge-chromatic number $\chi^{\prime}(G)$ of $G$ is the
least number of colors needed for a proper edge-coloring of $G$.
Clearly, $\chi^{\prime}(G)\geq \Delta(G)$ for every graph $G$. On
the other hand, the well-known theorem of Vizing \cite{b16} states
that the edge-chromatic number of any graph $G$ is either
$\Delta(G)$ or $\Delta(G)+1$. One of the most important, interesting
and long-standing problem in this field is the problem of
determining the exact value of the edge-chromatic number of graphs.
There are many results in this direction, in particular, the exact
value of the edge-chromatic number is known for bipartite graphs
\cite{b13}, complete graphs \cite{b4,b18}, complete multipartite
graphs \cite{b9,b14}, split graphs with odd maximum degree
\cite{b5}, outerplanar graphs \cite{b7}, planar graphs $G$ with
$\Delta(G)\geq 7$ \cite{b17,b20}.

A graph $G$ is a complete $k$-partite ($k\geq 2$) graph if its
vertices can be partitioned into $k$ independent sets
$V_{1},\ldots,V_{k}$ such that each vertex in $V_{i}$ is adjacent to
all the other vertices in $V_{j}$ for $1\leq i<j\leq k$. A complete
$k$-partite graph $G$ is a complete balanced $k$-partite graph if
$\vert V_{1}\vert =\vert V_{2}\vert =\cdots =\vert V_{k}\vert$.
Clearly, if $G$ is a complete balanced $k$-partite graph with $n$
vertices in each part, then $\Delta(G)=(k-1)n$. Note that the
complete graph $K_{n}$ and the complete balanced bipartite graph
$K_{n,n}$ are special cases of the complete balanced $k$-partite
graph. In \cite{b14}, Laskar and Hare proved the following:

\begin{theorem}
\label{mytheorem1} If $G$ is a complete balanced $k$-partite graph
with $n$ vertices in each part, then
\begin{center}
$\chi^{\prime}(G)=\left\{
\begin{tabular}{ll}
$(k-1)n$, & if $nk$ is even, \\
$(k-1)n+1$, & if $nk$ is odd. \\
\end{tabular}%
\right.$
\end{center}
\end{theorem}

A more general result was obtained by Hoffman and Rodger \cite{b9}.
Before we formulate this result we need a definition of the overfull
graph. A graph $G$ is overfull if $\vert
E(G)\vert>\left\lfloor\frac{\vert
V(G)\vert}{2}\right\rfloor\Delta(G)$. Clearly, if $G$ is overfull,
then $\chi^{\prime}(G)=\Delta(G)+1$.

\begin{theorem}
\label{mytheorem2} If $G$ is a complete $k$-partite graph, then
\begin{center}
$\chi^{\prime}(G)=\left\{
\begin{tabular}{ll}
$\Delta(G)$, & if $G$ is not overfull, \\
$\Delta(G)+1$, & if $G$ is overfull. \\
\end{tabular}%
\right.$
\end{center}
\end{theorem}

An edge-coloring of a graph $G$ with colors $1,\ldots,t$ is an
interval $t$-coloring if all colors are used, and the colors of
edges incident to each vertex of $G$ are distinct and form an
interval of integers. A graph $G$ is interval colorable if $G$ has
an interval $t$-coloring for some positive integer $t$. For an
interval colorable graph $G$, the least and the greatest values of
$t$ for which $G$ has an interval $t$-coloring are denoted by $w(G)$
and $W(G)$, respectively. The concept of interval edge-coloring was
introduced by Asratian and Kamalian \cite{b1}. In \cite{b1,b2}, they
proved the following:

\begin{theorem}
\label{mytheorem3} If $G$ is a regular graph, then

\begin{description}
\item[(1)] $G$ is interval colorable if and only if $\chi^{\prime}(G)=\Delta(G)$.

\item[(2)] If $G$ is interval colorable and $w(G)\leq t\leq W(G)$, then $G$
has an interval $t$-coloring.
\end{description}
\end{theorem}

In \cite{b10}, Kamalian investigated interval colorings of complete
bipartite graphs and trees. In particular, he proved the following:

\begin{theorem}
\label{mytheorem4} For any $r,s\in \mathbb{N}$, the complete
bipartite graph $K_{r,s}$ is interval colorable, and

\begin{description}
\item[(1)] $w\left(K_{r,s}\right)=r+s-\gcd(r,s)$,

\item[(2)] $W\left(K_{r,s}\right)=r+s-1$,

\item[(3)] if $w\left(K_{r,s}\right)\leq t\leq W\left(K_{r,s}\right)$, then $K_{r,s}$
has an interval $t$-coloring.
\end{description}
\end{theorem}

Later, Kamalian \cite{b11} obtained an upper bound on $W(G)$ for an
interval colorable graph $G$ depending on the number of vertices of
$G$.

\begin{theorem}
\label{mytheorem5} If $G$ is a connected interval colorable graph,
then $W(G)\leq 2\vert V(G)\vert -3$.
\end{theorem}

Clearly, this bound is sharp for the complete graph $K_{2}$, but if
$G\neq K_{2}$, then this upper bound can be improved to $2\left\vert
V(G)\right\vert -4$ \cite{b8}. For an $r$-regular graph $G$,
Kamalian and Petrosyan \cite{b12} showed that if $G$ with at least
$2r+2$ vertices admits an interval $t$-coloring, then $t\leq
2\left\vert V(G)\right\vert -5$. For a planar graph $G$, Axenovich
\cite{b3} showed that if $G$ has an interval $t$-coloring, then
$t\leq \frac{11}{6}\left\vert V(G)\right\vert$. In \cite{b15},
Petrosyan investigated interval colorings of complete graphs and
$n$-dimensional cubes. First note that $K_{2n+1}$ is not interval
colorable, but $K_{2n}$ is interval colorable and $w(K_{2n})=2n-1$
for any $n\in \mathbb{N}$. For $W(K_{2n})$, Petrosyan \cite{b15}
proved the following:

\begin{theorem}
\label{mytheorem6} If $n=p2^{q}$, where $p$ is odd and $q$ is
nonnegative, then $W\left(K_{2n}\right)\geq 4n-2-p-q$.
\end{theorem}

In this paper we investigate interval colorings of complete balanced
$k$-partite graphs. In particular, we generalize Theorem
\ref{mytheorem6} for complete balanced $k$-partite graphs.
Also, we discuss some other corollaries of our result.\\

\section{Results}\

If $\alpha $ is a proper edge-coloring of $G$ and $v\in V(G)$, then
$S\left(v,\alpha \right)$ denotes the set of colors appearing on
edges incident to $v$.

Let $[t]$ denote the set of the first $t$ natural numbers. Let
$\left\lfloor a\right\rfloor $ denote the largest integer less than
or equal to $a$. For two positive integers $a$ and $b$ with $a\leq
b$, the set $\left\{a,\ldots,b\right\}$ is denoted by $[a,b]$ and
called an interval. For an interval $[a,b]$ and a nonnegative number
$p$, the notation $[a,b]\oplus p$ means: $[a+p,b+p]$.\\

We also need the following lemma.

\begin{lemma}
\label{mylemma} If $K_{n,n}$ is a complete balanced bipartite graph
with a bipartition $(U,V)$, where $U=\{u_{1},\ldots,u_{n}\}$ and
$V=\{v_{1},\ldots,v_{n}\}$, then $K_{n,n}$ has an interval
$(2n-1)$-coloring $\alpha$ such that
$S(u_{i},\alpha)=S(v_{i},\alpha)=[i,i+n-1]$ for $1\leq i\leq n$
\end{lemma}
\begin{proof}
Let $(U,V)$ be a bipartition of $K_{n,n}$, where
$U=\{u_{1},\ldots,u_{n}\}$ and $V=\{v_{1},\ldots,v_{n}\}$. Define a
coloring $\alpha$ of the edges of $K_{n,n}$ as follows: for each
edge $u_{i}v_{j}\in E(K_{n,n})$, let $\alpha(u_{i}v_{j})=i+j-1$,
where $1\leq i\leq n, 1\leq j\leq n$. Clearly $\alpha$ is an
interval $(2n-1)$-coloring of $K_{n,n}$ and
$S(u_{i},\alpha)=S(v_{i},\alpha)=[i,i+n-1]$ for $1\leq i\leq
n$.~$\square$
\end{proof}

Let $G$ be a complete balanced $k$-partite graph with $n$ vertices
in each part. By Theorems \ref{mytheorem1} and \ref{mytheorem3}, we
have that $G$ is interval colorable if and only if $nk$ is even.
Moreover, if $nk$ is even, then $w(G)=\Delta(G)=(k-1)n$. On the
other hand, by Theorem \ref{mytheorem5}, we obtain $W(G)\leq 2nk-3$
whenever $nk$ is even. Now we derive a lower bound for $W(G)$.

\begin{theorem}
\label{mytheorem7} If $G$ is a complete balanced $k$-partite graph
with $n$ vertices in each part and $nk$ is even, then $W(G)\geq
\left(\frac{3}{2}k-1\right)n-1$.
\end{theorem}
\begin{proof} We distinguish our proof into two cases.

Case 1: $k$ is even.

Let $V(G)=\left\{v^{(i)}_{j}\colon\,1\leq i\leq k,1\leq j\leq
n\right\}$ and

$E(G)=\left\{v^{(i)}_{p}v^{(j)}_{q}\colon\,1\leq i<j\leq k,1\leq
p\leq n,1\leq q\leq n\right\}$.

Define an edge-coloring $\alpha$ of the graph $G$.\\

For each edge $v^{(i)}_{p}v^{(j)}_{q}\in E(G)$ with $1\leq i<j\leq
k$ and $p=1,\ldots,n$, $q=1,\ldots,n$, define a
color $\alpha\left(v^{(i)}_{p}v^{(j)}_{q}\right)$ as follows:\\

for $i=1,\ldots,\left\lfloor\frac{k}{4}\right\rfloor$, $
j=2,\ldots,\frac{k}{2}$, $i+j\leq \frac{k}{2}+1$, let

\begin{center}
$\alpha \left(v_{p}^{(i)}v_{q}^{(j)}\right)=\left(
i+j-3\right)n+p+q-1$;\\
\end{center}

for $i=2,\ldots,\frac{k}{2}-1$,
$j=\left\lfloor\frac{k}{4}\right\rfloor+2,\ldots,\frac{k}{2}$,
$i+j\geq \frac{k}{2}+2$, let

\begin{center}
$\alpha \left(v_{p}^{(i)}v_{q}^{(j)}\right)=\left(
i+j+\frac{k}{2}-4\right)n+p+q-1$;\\
\end{center}

for $i=3,\ldots,\frac{k}{2}$, $j=\frac{k}{2}+1,\dots,k-2$, $j-i\leq
\frac{k}{2}-2$, let

\begin{center}
$\alpha \left(v_{p}^{(i)}v_{q}^{(j)}\right)=\left(
\frac{k}{2}+j-i-1\right)n+p+q-1$;\\
\end{center}

for $i=1,\ldots,\frac{k}{2}$, $j=\frac{k}{2}+1,\ldots,k$, $j-i\geq
\frac{k}{2}$, let

\begin{center}
$\alpha
\left(v_{p}^{(i)}v_{q}^{(j)}\right)=\left(j-i-1\right)n+p+q-1$;\\
\end{center}

for $i=2,\ldots,1+\left\lfloor \frac{k-2}{4}\right\rfloor$,
$j=\frac{k}{2}+1,\dots,\frac{k}{2}+\left\lfloor\frac{k-2}{4}\right\rfloor
$, $j-i=\frac{k}{2}-1$, let

\begin{center}
$\alpha
\left(v_{p}^{(i)}v_{q}^{(j)}\right)=\left(2i-3\right)n+p+q-1$;\\
\end{center}

for $i=\left\lfloor
\frac{k-2}{4}\right\rfloor+2,\ldots,\frac{k}{2}$,
$j=\frac{k}{2}+1+\left\lfloor\frac{k-2}{4}\right\rfloor,\ldots,k-1$,
$j-i=\frac{k}{2}-1$, let

\begin{center}
$\alpha
\left(v_{p}^{(i)}v_{q}^{(j)}\right)=\left(i+j-3\right)n+p+q-1$;\\
\end{center}

for $i=\frac{k}{2}+1,\dots,\frac{k}{2}+\left\lfloor
\frac{k}{4}\right\rfloor-1$, $j=\frac{k}{2}+2,\ldots,k-2$, $i+j\leq
\frac{3}{2}k-1$, let

\begin{center}
$\alpha
\left(v_{p}^{(i)}v_{q}^{(j)}\right)=\left(i+j-k-1\right)n+p+q-1$;\\
\end{center}

for $i=\frac{k}{2}+1,\ldots,k-1$, $j=\frac{k}{2}+\left\lfloor
\frac{k}{4}\right\rfloor +1,\dots,k$, $i+j\geq \frac{3}{2}k$, let

\begin{center}
$\alpha \left(v_{p}^{(i)}v_{q}^{(j)}\right)=\left(
i+j-\frac{k}{2}-2\right)n+p+q-1$.\\
\end{center}

Let us prove that $\alpha$ is an interval
$\left(\left(\frac{3}{2}k-1\right)n-1\right)$-coloring of the graph
$G$.

First we show that for each $t\in
\left[\left(\frac{3}{2}k-1\right)n-1\right]$, there is an edge $e\in
E(G)$ with $\alpha(e)=t$.

Consider the vertices
$v_{1}^{(1)},\ldots,v_{n}^{(1)},v_{1}^{(k)},\ldots,v_{n}^{(k)}$.
Now, by Lemma \ref{mylemma} and the definition of $\alpha$, for
$1\leq j\leq n$,

\begin{center}
$S\left(v_{j}^{(1)},\alpha\right)
=\bigcup_{l=1}^{k-1}\left([j,j+n-1]\oplus (l-1)n\right)$ and
$S\left(v_{j}^{(k)},\alpha\right)
=\bigcup_{l=\frac{k}{2}}^{\frac{3}{2}k-2}\left([j,j+n-1]\oplus
(l-1)n\right)$.
\end{center}

Let $\overline{C}$ and $\overline{\overline{C}}$ be the subsets of
colors appear on edges incident to the vertices
$v_{1}^{(1)},\ldots,v_{n}^{(1)}$ and
$v_{1}^{(k)},\ldots,v_{n}^{(k)}$ in the coloring $\alpha$,
respectively, that is:

\begin{center}
$\overline{C}=\bigcup_{j=1}^{n} S\left(v_{j}^{(1)},\alpha\right)$
and $\overline{\overline{C}}=\bigcup_{j=1}^{n}
S\left(v_{j}^{(k)},\alpha\right)$.
\end{center}

It is straightforward to check that $\overline{C}\cup
\overline{\overline{C}}=\left[\left(\frac{3}{2}k-1\right)n-1\right]$,
so for each $t\in \left[\left(\frac{3}{2}k-1\right)n-1\right]$,
there is an edge $e\in E(G)$ with $\alpha(e)=t$.

Next we show that the edges incident to any vertex of $G$ are
colored by $(k-1)n$ consecutive colors.

Let $v_{j}^{(i)}\in V(G)$, where $1\leq i\leq k$, $1\leq j\leq n$.

Subcase 1.1. $1\leq i\leq 2$, $1\leq j\leq n$.

By Lemma \ref{mylemma} and the definition of $\alpha$, we have

\begin{center}
$S\left(v_{j}^{(1)},\alpha\right)=S\left(v_{j}^{(2)},\alpha\right)
=\bigcup_{l=1}^{k-1}\left([j,j+n-1]\oplus
(l-1)n\right)=\left[j,j+(k-1)n-1\right]$.
\end{center}

Subcase 1.2. $3\leq i\leq \frac{k}{2}$, $1\leq j\leq n$.

By Lemma \ref{mylemma} and the definition of $\alpha$, we have

\begin{center}
$S\left(v_{j}^{(i)},\alpha\right)
=\bigcup_{l=i-1}^{k-3+i}\left([j,j+n-1]\oplus
(l-1)n\right)=\left[j+(i-2)n,j+(k-3+i)n-1\right]$.
\end{center}

Subcase 1.3. $\frac{k}{2}+1\leq i\leq k-2$, $1\leq j\leq n$.

By Lemma \ref{mylemma} and the definition of $\alpha$, we have

\begin{center}
$S\left(v_{j}^{(i)},\alpha\right)
=\bigcup_{l=i-\frac{k}{2}+1}^{\frac{k}{2}-1+i}\left([j,j+n-1]\oplus
(l-1)n\right)=\left[j+\left(i-\frac{k}{2}\right)n,j+\left(\frac{k}{2}-1+i\right)n-1\right]$.
\end{center}

Subcase 1.4. $k-1\leq i\leq k,1\leq j\leq n$.

By Lemma \ref{mylemma} and the definition of $\alpha$, we have

\begin{center}
$S\left(v_{j}^{(k-1)},\alpha\right)=S\left(v_{j}^{(k)},\alpha\right)
=\bigcup_{l=\frac{k}{2}}^{\frac{3}{2}k-2}\left([j,j+n-1]\oplus
(l-1)n\right)
=\left[j+\left(\frac{k}{2}-1\right)n,j+\left(\frac{3k}{2}-2\right)n-1\right]$.
\end{center}

This shows that $\alpha$ is an interval
$\left(\left(\frac{3}{2}k-1\right)n-1\right)$-coloring of $G$; thus
$W(G)\geq \left(\frac{3}{2}k-1\right)n-1$ for even $k\geq 2$.

Case 2: $n$ is even.

Let $n=2m$, $m\in \mathbb{N}$. Let
$U_{i}=\left\{v_{1}^{(i)},\ldots,v_{m}^{(i)},v_{1}^{(k+i)},\ldots,v_{m}^{(k+i)}\right\}$
($1\leq i\leq k$) be the $k$ independent sets of vertices of $G$.
For $i=1,\ldots,2k$, define the set $V_{i}$ as follows:
$V_{i}=\left\{v_{1}^{(i)},\ldots,v_{m}^{(i)}\right\}$. Clearly,
$V(G)=\bigcup_{i=1}^{2k}V_{i}$. For $1\leq i<j\leq 2k$, define
$(V_{i},V_{j})$ as the set of all edges between $V_{i}$ and $V_{j}$.
It is easy to see that for $1\leq i<j\leq 2k$, $\left\vert
(V_{i},V_{j})\right\vert=m^{2}$ except for $\left\vert
(V_{i},V_{k+i})\right\vert=0$ whenever $i=1,\ldots,k$. If we
consider the sets $V_{i}$ as the vertices and the sets
$(V_{i},V_{j})$ as the edges, then we obtain that $G$ is isomorphic
to the graph $K_{2k}-F$, where $F$ is a perfect matching. Now we
define an edge-coloring $\beta$ of the graph $G$.

For each edge $v^{(i)}_{p}v^{(j)}_{q}\in E(G)$ with $1\leq i<j\leq
2k$ and $p=1,\ldots,m$, $q=1,\ldots,m$, define a
color $\beta\left(v^{(i)}_{p}v^{(j)}_{q}\right)$ as follows:\\

for $i=1,\ldots,\left\lfloor\frac{k}{2}\right\rfloor$, $
j=2,\ldots,k$, $i+j\leq k+1$, let

\begin{center}
$\beta \left(v_{p}^{(i)}v_{q}^{(j)}\right)=\left(
i+j-3\right)m+p+q-1$;\\
\end{center}

for $i=2,\ldots,k-1$,
$j=\left\lfloor\frac{k}{2}\right\rfloor+2,\ldots,k$, $i+j\geq k+2$,
let

\begin{center}
$\beta\left(v_{p}^{(i)}v_{q}^{(j)}\right)=\left(
i+j+k-5\right)m+p+q-1$;\\
\end{center}

for $i=3,\ldots,k$, $j=k+1,\dots,2k-2$, $j-i\leq k-2$, let

\begin{center}
$\beta\left(v_{p}^{(i)}v_{q}^{(j)}\right)=\left(
k+j-i-2\right)m+p+q-1$;\\
\end{center}

for $i=1,\ldots,k-1$, $j=k+2,\ldots,2k$, $j-i\geq k+1$, let

\begin{center}
$\beta
\left(v_{p}^{(i)}v_{q}^{(j)}\right)=\left(j-i-2\right)m+p+q-1$;\\
\end{center}

for $i=2,\ldots,1+\left\lfloor \frac{k-1}{2}\right\rfloor$,
$j=k+1,\dots,k+\left\lfloor\frac{k-1}{2}\right\rfloor $, $j-i=k-1$,
let

\begin{center}
$\beta
\left(v_{p}^{(i)}v_{q}^{(j)}\right)=\left(2i-3\right)m+p+q-1$;\\
\end{center}

for $i=\left\lfloor \frac{k-1}{2}\right\rfloor+2,\ldots,k$,
$j=k+1+\left\lfloor\frac{k-1}{2}\right\rfloor,\ldots,2k-1$,
$j-i=k-1$, let

\begin{center}
$\beta
\left(v_{p}^{(i)}v_{q}^{(j)}\right)=\left(i+j-4\right)m+p+q-1$;\\
\end{center}

for $i=k+1,\dots,k+\left\lfloor\frac{k}{2}\right\rfloor-1$,
$j=k+2,\ldots,2k-2$, $i+j\leq 3k-1$, let

\begin{center}
$\beta
\left(v_{p}^{(i)}v_{q}^{(j)}\right)=\left(i+j-2k-1\right)m+p+q-1$;\\
\end{center}

for $i=k+1,\ldots,2k-1$, $j=k+\left\lfloor \frac{k}{2}\right\rfloor
+1,\dots,2k$, $i+j\geq 3k$, let

\begin{center}
$\beta \left(v_{p}^{(i)}v_{q}^{(j)}\right)=\left(
i+j-k-3\right)m+p+q-1$.\\
\end{center}

Let us prove that $\beta$ is an interval
$\left(\left(\frac{3}{2}k-1\right)n-1\right)$-coloring of the graph
$G$.

First we show that for each $t\in
\left[\left(\frac{3}{2}k-1\right)n-1\right]$, there is an edge $e\in
E(G)$ with $\beta(e)=t$.

Consider the vertices
$v_{1}^{(1)},\ldots,v_{m}^{(1)},v_{1}^{(2k)},\ldots,v_{m}^{(2k)}$.
Now, by Lemma \ref{mylemma} and the definition of $\beta$, for
$1\leq j\leq m$,

\begin{center}
$S\left(v_{j}^{(1)},\beta\right)
=\bigcup_{l=1}^{2k-2}\left([j,j+m-1]\oplus (l-1)m\right)$ and
$S\left(v_{j}^{(2k)},\beta\right)
=\bigcup_{l=k}^{3k-3}\left([j,j+m-1]\oplus (l-1)m\right)$.
\end{center}

Let $\tilde{C}$ and $\tilde{\tilde{C}}$ be the subsets of colors
appear on edges incident to the vertices
$v_{1}^{(1)},\ldots,v_{m}^{(1)}$ and
$v_{1}^{(2k)},\ldots,v_{m}^{(2k)}$ in the coloring $\beta$,
respectively, that is:

\begin{center}
$\tilde{C}=\bigcup_{j=1}^{m} S\left(v_{j}^{(1)},\beta\right)$ and
$\tilde{\tilde{C}}=\bigcup_{j=1}^{m}
S\left(v_{j}^{(2k)},\beta\right)$.
\end{center}

It is straightforward to check that $\tilde{C}\cup
\tilde{\tilde{C}}=\left[\left(\frac{3}{2}k-1\right)n-1\right]$, so
for each $t\in \left[\left(\frac{3}{2}k-1\right)n-1\right]$, there
is an edge $e\in E(G)$ with $\beta(e)=t$.

Next we show that the edges incident to any vertex of $G$ are
colored by $(k-1)n$ consecutive colors.

Let $v_{j}^{(i)}\in V(G)$, where $1\leq i\leq 2k$, $1\leq j\leq m$.

Subcase 2.1. $1\leq i\leq 2$, $1\leq j\leq m$.

By Lemma \ref{mylemma} and the definition of $\beta$, we have

\begin{center}
$S\left(v_{j}^{(1)},\beta\right)=S\left(v_{j}^{(2)},\beta\right)
=\bigcup_{l=1}^{2k-2}\left([j,j+m-1]\oplus
(l-1)m\right)=\left[j,j+(2k-2)m-1\right]$.
\end{center}

Subcase 2.2. $3\leq i\leq k$, $1\leq j\leq m$.

By Lemma \ref{mylemma} and the definition of $\beta$, we have

\begin{center}
$S\left(v_{j}^{(i)},\beta\right)
=\bigcup_{l=i-1}^{2k-4+i}\left([j,j+m-1]\oplus
(l-1)m\right)=\left[j+(i-2)m,j+(2k-4+i)m-1\right]$.
\end{center}

Subcase 2.3. $k+1\leq i\leq 2k-2$, $1\leq j\leq m$.

By Lemma \ref{mylemma} and the definition of $\beta$, we have

\begin{center}
$S\left(v_{j}^{(i)},\beta\right)
=\bigcup_{l=i-k+1}^{k-2+i}\left([j,j+m-1]\oplus
(l-1)m\right)=\left[j+(i-k)m,j+(k-2+i)m-1\right]$.
\end{center}

Subcase 2.4. $2k-1\leq i\leq 2k,1\leq j\leq m$.

By Lemma \ref{mylemma} and the definition of $\beta$, we have

\begin{center}
$S\left(v_{j}^{(2k-1)},\beta\right)=S\left(v_{j}^{(2k)},\beta\right)
=\bigcup_{l=k}^{3k-3}\left([j,j+m-1]\oplus (l-1)m\right)
=\left[j+(k-1)m,j+(3k-3)m-1\right]$.
\end{center}

This shows that $\beta$ is an interval
$\left(\left(\frac{3}{2}k-1\right)n-1\right)$-coloring of $G$; thus
$W(G)\geq \left(\frac{3}{2}k-1\right)n-1$ for even $n\geq 2$.
~$\square$
\end{proof}

From Theorems \ref{mytheorem3}(1) and \ref{mytheorem7}, and taking
into account that a complete balanced $k$-partite graph $G$ with $n$
vertices in each part is overfull when $nk$ is odd, we have:

\begin{corollary}
\label{mycorollary1} If $G$ is a complete balanced $k$-partite graph
with $n$ vertices in each part, then $\chi^{\prime}(G)=(k-1)n$ if
and only if $nk$ is even.
\end{corollary}

From Theorems \ref{mytheorem3}(2) and \ref{mytheorem7}, we have:

\begin{corollary}
\label{mycorollary2} Let $G$ be a complete balanced $k$-partite
graph with $n$ vertices in each part and $nk$ is even. If
$(k-1)n\leq t\leq \left(\frac{3}{2}k-1\right)n-1$, then $G$ has an
interval $t$-coloring.
\end{corollary}

Also, note that the proof of the case 2 implies that if a graph $G$
with $n$ vertices is $(n-2)$-regular, then $\chi^{\prime}(G)=n-2$.\\

The next theorem improves the lower bound in Theorem
\ref{mytheorem7} for complete balanced $k$-partite graphs with even
$k$.

\begin{theorem}
\label{mytheorem8} Let $G$ be a complete balanced $k$-partite graph
with $n$ vertices in each part. If $k=p2^{q}$, where $p$ is odd and
$q\in \mathbb{N}$, then $W(G)\geq (2k-p-q)n-1$.
\end{theorem}
\begin{proof}
Let $V(G)=\left\{v^{(i)}_{j}\colon\,1\leq i\leq k,1\leq j\leq
n\right\}$ and $V(K_{k})=\{u_{1},\ldots,u_{k}\}$. Also, let
$E(G)=\left\{v^{(i)}_{r}v^{(j)}_{s}\colon\,1\leq i<j\leq k,1\leq
r\leq n,1\leq s\leq n\right\}$ and
$E(K_{k})=\{u_{i}u_{j}\colon\,1\leq i<j\leq k\}$.

Since $k=p2^{q}$, where $p$ is odd and $q\in \mathbb{N}$, by Theorem
\ref{mytheorem6}, there exists an interval $(2k-1-p-q)$-coloring
$\alpha$ of $K_{k}$. Now we define an edge-coloring $\beta$ of the
graph $G$.

For each edge $v^{(i)}_{r}v^{(j)}_{s}\in E(G)$ with $1\leq i<j\leq
k$ and $r=1,\ldots,n$, $s=1,\ldots,n$, define a color
$\beta\left(v^{(i)}_{r}v^{(j)}_{s}\right)$ as follows:

\begin{center}
$\beta\left(v^{(i)}_{r}v^{(j)}_{s}\right)=\left(\alpha\left(u_{i}u_{j}\right)-1\right)n+r+s-1$.
\end{center}

By Lemma \ref{mylemma} and the definition of $\beta$, and taking
into account that $\max S\left(u_{i},\alpha\right)-\min
S\left(u_{i},\alpha\right)=k-2$ for $i=1,\ldots,k$, we have

\begin{eqnarray*}
S\left(v_{j}^{(i)},\beta\right)&=&\bigcup_{l=\min
S(u_{i},\alpha)}^{\max S(u_{i},\alpha)}\left([j,j+n-1]\oplus
(l-1)n\right)=\\
&=&[j+\left(\min S(u_{i},\alpha)-1\right)n,j+\max
S(u_{i},\alpha)n-1]
\end{eqnarray*}
for $i=1,\ldots,k$ and $j=1,\ldots,n$, and
\begin{eqnarray*}
\bigcup_{i=1}^{k}\bigcup_{j=1}^{n}S\left(v_{j}^{(i)},\beta\right)&=&[(2k-p-q)n-1].
\end{eqnarray*}

This shows that $\beta$ is an interval $((2k-p-q)n-1)$-coloring of
the graph $G$; thus $W(G)\geq (2k-p-q)n-1$. ~$\square$
\end{proof}

From Theorems \ref{mytheorem3}(2) and \ref{mytheorem8}, we have:

\begin{corollary}
\label{mycorollar3} Let $G$ be a complete balanced $k$-partite graph
with $n$ vertices in each part and $k=p2^{q}$, where $p$ is odd and
$q\in \mathbb{N}$. If $(k-1)n\leq t\leq (2k-p-q)n-1$, then $G$ has
an interval $t$-coloring.
\end{corollary}

\section{Problems}\

In the previous section we obtained some results on interval
colorings of complete balanced multipartite graphs, but very small
is known about interval colorings of complete unbalanced
multipartite graphs. In fact, there are only two results on interval
colorings of complete unbalanced multipartite graphs. Let
$n_{1}\leq\cdots \leq n_{k}$ be positive integers. The complete
multipartite graph $K_{n_{1},\ldots,n_{k}}$ is a complete
$k$-partite graph for which $\vert V_{i}\vert =n_{i}$,
$i=1,\ldots,k$. The first result is Theorem \ref{mytheorem4} which
gives all possible values of the number of colors in interval
colorings of $K_{n_{1},n_{2}}$. The second result was obtained by
Feng and Huang \cite{b6}. In \cite{b6}, they proved that the
complete $3$-partite graph $K_{1,1,n}$ is interval colorable if and
only if $n$ is even. Now we would like to formulate some problems on
interval colorings of complete multipartite graphs:

\begin{problem}\label{myproblem1}
Characterize all interval colorable complete multipartite graphs.
\end{problem}

\begin{problem}\label{myproblem2}
Find the exact values of $w(G)$ and $W(G)$ for interval colorable
complete multipartite graphs G.
\end{problem}

\begin{problem}\label{myproblem3}
Find the exact value of $W\left(K_{n,\ldots,n}\right)$ for interval
colorable complete balanced $k$-partite graphs $K_{n,\ldots,n}$.
\end{problem}

Note that even a special case of Problem \ref{myproblem3} is open:
the problem of determining the exact value of $W(K_{2n})$ for
complete graph $K_{2n}$.

\end{document}